\documentclass[11pt,reqno]{amsart}
\usepackage{color}
\usepackage[english]{babel}
\usepackage{amsfonts}
\usepackage{amsmath}
\usepackage{amsthm}
\usepackage{amssymb}
\usepackage[misc]{ifsym}
\usepackage{bbm}
\usepackage{mathrsfs}
\usepackage{esint}
\usepackage{faktor}
\usepackage[utf8]{inputenc}
\usepackage{afterpage}
\usepackage[left=2.9cm,right=2.9cm,top=2.8cm,bottom=2.8cm]{geometry}
\usepackage{graphicx}
\usepackage{enumitem}
\usepackage[dvipsnames]{xcolor}
\usepackage[colorlinks=true,urlcolor=blue,citecolor=blue,linkcolor=blue,linktocpage,pdfpagelabels,bookmarksnumbered,bookmarksopen]{hyperref}

\newcommand{\N}{\mathbb{N}}

\newcommand{\R}{\mathbb{R}}

\newcommand{\D}{\mathcal{D}}
\renewcommand{\H}{\mathcal{H}}
\newcommand{\Div}{\mathrm{div} \, }
\newcommand{\dx}{\, {\rm d} x}

\newcommand{\dt}{\, {\rm d} t}

\newcommand{\dxi}{\, {\rm d} \xi}
\newcommand{\dH}{\, {\rm d} \mathcal{H}}
\newcommand{\eps}{\varepsilon}
\renewcommand{\phi}{\varphi}
\newcommand{\loc}{{\rm loc}}

\newtheorem{lemma}{Lemma}[section]
\newtheorem{thm}[lemma]{Theorem}
\newtheorem{prop}[lemma]{Proposition}

\theoremstyle{definition}
\newtheorem{defi}[lemma]{Definition}
\newtheorem{rmk}[lemma]{Remark}
\newtheorem{ex}[lemma]{Example}

\numberwithin{equation}{section}
\DeclareMathOperator*{\esssup}{ess \, sup}
\DeclareMathOperator*{\essinf}{ess \, inf}

\begin{document}

\title[Strong solutions to singular discontinuous $p$-Laplacian problems]{Strong solutions to singular \\ discontinuous $p$-Laplacian problems}

\author[U. Guarnotta]{Umberto Guarnotta}
\address[U. Guarnotta]{Dipartimento di Ingegneria Industriale e Scienze Matematiche, Università Politecnica delle Marche, Via Brecce Bianche 12,
60131 Ancona, Italy}
\email{u.guarnotta@univpm.it}

\author[S.A. Marano]{Salvatore A. Marano}
\address[S.A. Marano]{Dipartimento di Matematica e Informatica, Universit\`a degli Studi di Catania, Viale A. Doria 6, 95125 Catania, Italy}
\email{marano@dmi.unict.it}

\begin{abstract}
In this paper, the existence of positive strong solutions to a Dirichlet $p$-Laplacian problem with reaction both singular at zero and highly discontinuous is investigated. In particular, it is only required that the set of discontinuity points has Lebesgue measure zero.
\end{abstract}

\maketitle

{
\let\thefootnote\relax
\footnote{\Letter \quad Corresponding author: Umberto Guarnotta (u.guarnotta@univpm.it).}
}
\setcounter{footnote}{0}

%
\section{Introduction}
Let $\Omega$ be a bounded domain in $\R^N$, $N\geq 2$, with a $C^2$ boundary $\partial\Omega$ and let $1<p<N$. This paper investigates the problem
\begin{equation}\label{problem}
\tag{P}
\left\{
\begin{alignedat}{2}
-\Delta_p u & =f(u)\quad && \mbox{in}\;\;\Omega, \\
u & >0 \quad && \mbox{in}\;\;\Omega, \\
u & =0 \quad &&\mbox{on}\;\;\partial\Omega,
\end{alignedat}
\right.
\end{equation}
where $\Delta_p$ stands for the usual $p$-Laplace operator while $f:\R^+\to\R^+_0$ is a Borel measurable function satisfying the conditions below, collectively denoted by \hyperlink{Hf}{$({\rm H}_f)$}. Hereafter,
$$\underline{f}(s):=\lim_{\delta\to 0^+}\essinf_{|t-s|<\delta} f(t)\quad\mbox{and}\quad
\overline{f}(s):=\lim_{\delta\to 0^+}\esssup_{|t-s|<\delta} f(t)\quad\forall\, s\in\R^+.$$
\begin{enumerate}
[
label={$
{\rm (\roman*)}$},
ref={$({\rm H}_f){\rm (\roman*)}$}
]
\hypertarget{Hf}{}
\item \label{localbound} $f\in L^\infty_\loc(\R^+)$.
\vskip5pt
\item \label{singular}
There exists $\gamma\in(0,1)$ such that $\displaystyle{\limsup_{s\to 0^+}}\, s^\gamma f(s)< +\infty$.
\vskip0pt
\item \label{subsol}
$\displaystyle{\liminf_{s\to 0^+}}\frac{f(s)}{s^{p-1}}>\lambda_1$, with $\lambda_1>0$ being the first eigenvalue of $(-\Delta_p,W^{1,p}_0(\Omega))$.
\vskip2pt
\item\label{sublinear}
$\displaystyle{\limsup_{s\to+\infty}}\frac{f(s)}{s^{p-1}}<\lambda_1$.
\vskip5pt
\item \label{zeromeasure}
If $\D_f\subseteq\R^+$ indicates the set of discontinuity points for $f$ then $|\D_f|=0$.
\vskip5pt
\item \label{fzero}
$\underline{f}(s)=0,\;\; s\in\R^+\implies f(s)=0$.
\end{enumerate}
Hence, reactions we treat
\begin{itemize}
\item can be {\it weakly singular at zero}. As an example, $f(s):=\frac{1}{s^\gamma}$, $s\in\R^+$, fulfills \hyperlink{Hf}{$({\rm H}_f)$}.
\item may have a {\it `wide' set of discontinuities}, on condition that its measure is zero. In particular, even for non-singular $f$, this prevents using the classical dual method \cite{AB}.
\end{itemize}
The literature devoted to singular problems looks by now daily increasing; let us mention the monographs \cite{GR,Mo} as well as the recent survey \cite{GLM}. Contrariwise, the current production on differential equations with (possibly highly) discontinuous reactions is decidedly less rich and the best period probably was from the 1980s to the 2010s; see the monographs \cite{MR,GP2,CLM}, besides the survey \cite{LM}. Thus, trying to study problems with both issues, either in a bounded domain - like here - or in the whole space, might be of some interest. Actually, there are already a few related results, but for singular nonlinearities discontinuous at just one point \cite{DPTS,RSSC,SPC}.

Let us now recall the two notions of solution we will employ.
\begin{defi}
A function $u\in W^{1,p}_0(\Omega)$ is called a (weak) solution to \eqref{problem} when $u>0$ in $\Omega$ and
\begin{equation*}
\int_\Omega |\nabla u|^{p-2}\nabla u\nabla \phi \dx = \int_\Omega f(u)\phi \dx \quad \forall\, \phi\in W^{1,p}_0(\Omega),
\end{equation*}
with finite integrals.
\end{defi}
\begin{defi}
We say that $u\in W^{1,p}_0(\Omega)$ is a strong solution of \eqref{problem} if $u>0$ in $\Omega$,
$$|\nabla u|^{p-2}\nabla u \in W^{1,2}_\loc(\Omega,\R^N),$$
and $-\Delta_p u(x) = f(u(x))$ for almost every $x\in\Omega$.
\end{defi}
Evidently, `strong' implies `weak' but not viceversa. The main result of this paper is the following.
\begin{thm}\label{mainthm}
Let \hyperlink{Hf}{$({\rm H}_f)$} be satisfied. Then problem \eqref{problem} admits at least a strong solution $u\in C^{1,\alpha}(\overline{\Omega})$.
\end{thm}
We think useful to shortly outline the proof. Condition \ref{subsol} yields a positive sub-solution $\underline{u}\in C^{1,\alpha}_0(\overline{\Omega})$ of \eqref{problem}, that allows to suitably truncate the right-hand side $f$. We next associate with \eqref{problem} a family of regularized problems \eqref{auxprob} depending on a parameter $\eps\in (0,1)$. Under hypotheses \ref{localbound}--\ref{sublinear}, each \eqref{auxprob} possesses a solution $u_\eps\in W^{1,p}_0(\Omega)$; cf. Theorem \ref{solauxprob}. If $\eps:=1/n$ and $u_n:=u_{1/n}$, $n\in\N$, then Lemma \ref{solincl} ensures that
\begin{itemize}
\item up to sub-sequences, $u_n\rightharpoonup u$ in $W^{1,p}_0(\Omega)$,
\item $-\Delta_p u=v\in L^q(\Omega)_+\cap W^{-1,p'}(\Omega)$ for appropriate $q>1$, and
\item $\underline{f}(u)\leq v\leq\overline{f}(u)$ in $\Omega$, namely $-\Delta_p u\in\partial F(u)$ in $\Omega$, where $\partial F$ denotes the Clarke sub-differential of $F(s):=\int_0^s f(t)\dt$.
\end{itemize} 
Moreover, $u\in C^{1,\alpha}(\overline{\Omega})$ by Lemma \ref{regularity}. To show that $u$ strongly solves \eqref{problem} we combine assumptions \ref{zeromeasure}--\ref{fzero} with the so-called strong locality property of the $p$-Laplacian for Sobolev functions. Recall that
\begin{defi}
If $X,Y\subseteq L^1_\loc(\Omega)$ are two Banach spaces and $A$ is a map from $X$ to the space of distributions over $\Omega$, then $A$ is called strongly local of $X$ into $Y$ provided every distributional solution $u\in X$ of $A(u)=v$, with $v\in Y$, satisfies
$$ A(u) = 0 \quad\mbox{a.e. in}\;\; u^{-1}(c) \quad \mbox{for all}\;\; c\in\R.$$
\end{defi}
\begin{ex}\label{slex}
Here are some meaningful special cases.
\begin{enumerate}[label={(\arabic*)},ref={(\arabic*)}]
\item By Stampacchia's theorem \cite[Theorem 4.4 (iv), p. 153]{EG}, weak derivatives turn out strongly local operators from $W^{1,1}_\loc(\Omega)$ to $L^1_\loc(\Omega)$.
\item \label{CZ} The Laplacian is strongly local of $W^{1,2}(\Omega)$ into $L^p(\Omega)$ for all $p>1$. This is a consequence of Calderon-Zygmund's theorem \cite[Theorem 9.9]{GT}. 
\item By \cite[Proposition 4.2]{ABC}, the Laplacian turns out strongly local from $W^{1,2}(\Omega)$ to the space of Radon measures.
\item The distributional Laplacian is strongly local (in a suitable sense) of $L^1_\loc(\Omega)$ into the space of Radon measures; see \cite[Theorem 1.1]{APR} and \cite[Section 3]{MM}.
\item The divergence is not a strongly local operator from $C^0(\Omega,\R^N)$ to $L^\infty(\Omega)$; cf. \cite[Section 5]{ABC2} and \cite[Example 1.1]{APR}.
\item The $p$-Lapacian turns out strongly local of $W^{1,p}(\Omega)$ into $L^2_\loc(\Omega)$; cf. Proposition \ref{locality} below.
\end{enumerate}
\end{ex}
Finally, it is worth pointing out that our results hold true, with the same proofs, even when the $p$-Laplacian is replaced by the $(p,q)$-Laplacian (or the more general non-homogeneous operator called $a$-Laplacian; see, e.g., \cite[Appendix I]{GG}), provided the stronger condition
$$ \liminf_{s\to 0^+} f(s) > 0 $$
substitutes \ref{subsol}. In fact, this assumption comes into play only in the construction of the sub-solution (Lemma \ref{subsollemma}), which can now be obtained arguing as in \cite[Lemma 2.7]{CGL} (cf. also \cite[Lemma 3.3 and Theorem 3.5]{CGL} for the adaption of Lemma \ref{regularity} below).

\section{Preliminaries}
Henceforth, $\Omega$ is a bounded domain of the real Euclidean $N$-space $(\mathbb{R}^N,|\cdot|)$, $N\geq 2$, with a $C^2$-boundary $\partial\Omega$, while $\nu(x)$ indicates the outward unit normal vector to $\partial\Omega $ at its point $x$. Denote by $d:\overline{\Omega}\to[0,+\infty)$ the distance function of $\Omega$, i.e., $d(x):={\rm dist}(x,\partial\Omega)$ for all $x\in\overline{\Omega}$. Given $\delta>0$, set
\begin{equation*}
\Omega_{\delta }:=\{x\in \Omega:d(x)<\delta\}.
\end{equation*}
If $X(\Omega)$ is a real-valued function space on $\Omega$ and $u,v\in X(\Omega)$ then $u\leq v$ means $u(x)\leq v(x)$ a.e. in $\Omega$. Analogously for $u<v$, etc. To shorten notation, define
\begin{equation*}
\Omega (u\leq v):=\{x\in \Omega :u(x)\leq v(x)\},\quad X(\Omega)_+:=\{w\in X(\Omega): w\geq 0\}.
\end{equation*}
The symbol $|E|$ indicates the $N$-dimensional Lebesgue measure of $E\subseteq \mathbb{R}^{N}$, $\chi_E$ denotes its characteristic function,
$$t_\pm:=\max\{\pm t,0\},\quad t\in\R,$$
and $C$, $C_1$, etc. are positive constants, which may change in value from line to line, whose dependencies are specified when necessary.

Let $1<p<N$. We write $p'$ for the conjugate exponent of $p$ while $p^*:=\frac{Np}{N-p}$. The Sobolev space $W^{1,p}_0(\Omega)$ is equipped with the Poincaré norm 
\begin{equation*}
\Vert u\Vert_{1,p}:=\Vert |\nabla u|\Vert_p,\quad u\in W^{1,p}_0(\Omega),
\end{equation*}
where, as usual,
\begin{equation*}
\Vert v\Vert_q:=\left\{ 
\begin{array}{ll}
\left(\int_{\Omega }|v(x)|^q\dx\right)^{\frac{1}{q}} & \text{ if }1\leq q<+\infty, \\ 
\phantom{} &  \\ 
\underset{x\in\Omega}{\esssup}\, |v(x)| & \text{ when } q=+\infty.
\end{array}
\right.
\end{equation*}
Moreover, $W^{-1,p'}(\Omega)$ indicates the dual space of $W^{1,p}_0(\Omega )$,
\begin{equation*}
C^{1,\tau}_0(\overline{\Omega}):=\{u\in C^{1,\tau}(\overline{\Omega}): u\lfloor_{\partial\Omega}=0\},\quad 0<\tau<1,
\end{equation*}
and $v\in L^q(\Omega)\cap W^{-1,p'}(\Omega)$, $q\geq 1$, means that the linear map $u\mapsto \int_\Omega uv\dx$ generated by the function $v\in L^q(\Omega)$ turns out continuous on $W^{1,p}_0(\Omega)$.
\begin{prop}\label{distsumm}
If $0<\gamma<1<q<\frac{1}{\gamma}$ then $d^{-\gamma}\in L^q(\Omega)$.
\end{prop}
\begin{proof}
The conclusion is equivalent to the claim $d^{-r}\in L^1(\Omega)$ for all $r\in(0,1)$. Since $\partial\Omega\in C^2$, Lemma 14.16 in \cite{GT} furnishes $\delta>0$ such that $d\in C^2(K)$, with $K:=\overline{\Omega_\delta}$. In particular, by Weierstrass' theorem,
\begin{equation}\label{lapldbound}
|\Delta d(x)|\leq C\quad\forall\, x\in K.
\end{equation}
Taking a smaller $\delta$ if necessary, one has $|\nabla d|=1$ on $K$ (vide \cite[Theorem 3.14]{EG}) as well as $\nabla d$ orthogonal to both $\partial\Omega$ and $\Omega(d=t)$ for every $t\in (0,\delta]$. Through the divergence theorem we thus obtain
\begin{equation}\label{divthm}
\begin{aligned}
\int_{\Omega(d<t)}\Delta d \dx & =\int_{\Omega(d=t)} |\nabla d|^2 \dH^{N-1}- \int_{\partial\Omega} |\nabla d|^2 \dH^{N-1} \\
& =\H^{N-1}(\Omega(d=t))-\H^{N-1}(\partial\Omega),\quad t\in (0,\delta],
\end{aligned}
\end{equation}
where $\H^{N-1}$ denotes the $(N-1)$-dimensional Hausdorff measure. From \eqref{lapldbound}--\eqref{divthm} it follows
\begin{equation*}
\H^{N-1}(\Omega(d=t))=\H^{N-1}(\partial\Omega)+\int_{\Omega(d<t)}\Delta d \dx 
\leq\H^{N-1}(\partial\Omega) + C|\Omega|=:M \quad \forall\, t\in (0,\delta].
\end{equation*}
Now, if $r\in (0,1)$ then the coarea formula \cite[Theorem 3.13(ii)]{EG} entails
\begin{equation*}
\begin{aligned}
\int_\Omega d^{-r} \dx &= \int_K d^{-r}\dx+\int_{\Omega\setminus K}d^{-r}\dx
\leq\int_0^\delta t^{-r} \H^{N-1}(\Omega(d=t))\dt+\delta^{-r}|\Omega| \\
&\leq \frac{M}{1-r}\delta^{1-r}+\delta^{-r}|\Omega|<+\infty,
\end{aligned}
\end{equation*}
concluding the proof.
\end{proof}
The next result is well known; see e.g. \cite[Theorem 21.3]{OK}.
\begin{prop}[Hardy-Sobolev inequality]\label{hardysobolev}
Let $0<\gamma<1<p<N$. Then there exists $K>0$ such that
$$\int_\Omega d^{-\gamma}|u|\dx\leq K\|u\|_{1,p}\quad\forall\, u\in W^{1,p}_0(\Omega).$$
\end{prop}
Although the two properties below are folklore, we shall prove them.
\begin{prop}\label{distcomparison}
Suppose $u\in C^{1,\alpha}_0(\overline{\Omega })$. Then:
\begin{itemize}
\item[$({\rm a}_1)$] $\big\Vert d^{-1} u\big\Vert_{C^{0,\beta}(\overline{\Omega})}\leq C\Vert u\Vert_{C^{1,\alpha}(\overline{\Omega})}$, where $\beta:=\frac{\alpha}{\alpha+1}$ and $C>0$ does not depend on $u$. 
\item[$({\rm a}_2)$] If $u>0$ in $\Omega$ then there exists $l>0$ such that $u(x)\geq ld(x)$ for every $x\in\Omega$.
\end{itemize}
\end{prop}
\begin{proof}
First of all, observe that $u$ is Lipschitz continuous and one has
\begin{equation}\label{Lipu}
|u(x)|\leq \mathrm{Lip}(u)\,d(x)\quad \forall \,x\in \overline{\Omega }.
\end{equation}
The regularity of $\partial\Omega$ yields $\delta\in ]0,1[$, $\Pi\in C^{1}(\Omega_{\delta}, \partial\Omega)$ fulfilling 
\begin{equation}\label{propPi}
d(x)=|x-\Pi (x)|,\;\;\frac{x-\Pi (x)}{|x-\Pi (x)|}=-\nu(\Pi (x)),\;\;]\Pi
(x),x]\subseteq \Omega ,\;\;x\in \Omega _{\delta};  
\end{equation}
cf. \cite[Lemma 14.16]{GT} or \cite{F}. Conclusion $({\rm a}_1)$ easily follows once we achieve, for some $C_1>0$,
\begin{equation}\label{found}
\sup \left\{ \frac{\Big\vert\frac{u(x)}{d(x)}-\frac{u(y)}{d(y)}\Big\vert}{|x-y|^{\beta }} :x,y\in \Omega ,\,
0<|x-y|<\frac{\delta }{2}\right\}\leq C_{1}\,\Vert u\Vert_{C^{1,\alpha}
(\overline{\Omega })}.  
\end{equation}
So, pick $x,y\in \Omega $ such that $0<|x-y|<\frac{\delta }{2}$. If $\max\{d(x),d(y)\}\geq \delta $ then $x,y\in \Omega\setminus \Omega_{\delta /2}$. Consequently, 
\begin{equation*}
\sup_{x\in\Omega\setminus \Omega_{\delta /2}}\Big\vert\nabla\left(\frac{u(x)}{d(x)}\right)\Big\vert
\leq 2\,\frac{\mathrm{Lip}(u)}{\delta }+4\,\frac{\Vert u\Vert_{\infty }}{\delta ^{2}}
\leq \left( \frac{2}{\delta }+\frac{4}{\delta ^{2}}\right) \Vert u\Vert _{C^{1}(\overline{\Omega })},
\end{equation*}
because $d$ is 1-Lipschitz, and the mean value theorem entails 
\begin{equation}\label{casezero}
\frac{\Big\vert\frac{u(x)}{d(x)}-\frac{u(y)}{d(y)}\Big\vert}{|x-y|^{\beta}}\leq C_{2}\Vert u\Vert _{C^{1}(\overline{\Omega })}.  
\end{equation}
Assume now $d(y)\leq d(x)<\delta$; a similar argument applies when $d(x)\leq d(y)<\delta$. Two situations may occur.
\newline
1) $d(x)\leq |x-y|^{\frac{1}{\alpha+1}}$. Through the above-mentioned result again, besides \eqref{propPi}, we obtain
\begin{eqnarray*}
\frac{u(x)}{d(x)} &=&\frac{u(x)-u(\Pi(x))}{|x-\Pi(x)|}=-\nabla u(\hat{x})\nu(\Pi(x)),\\
\frac{u(y)}{d(y)} &=&\frac{u(y)-u(\Pi (y))}{|y-\Pi (y)|}=-\nabla u(\hat{y})\nu(\Pi (y))
\end{eqnarray*}
with appropriate $\hat{x}\in]\Pi(x),x[$, $\hat{y}\in]\Pi(y),y[$. This immediately leads to 
\begin{equation*}
\begin{split}
\Big\vert\frac{u(x)}{d(x)}-\frac{u(y)}{d(y)}\Big\vert & \leq |\nabla u(\hat{x})-\nabla u(\hat{y})|+|\nabla u(\hat{y})|\,|\nu(\Pi (x))-\nu(\Pi (y))| \\
& \leq \Vert u\Vert _{C^{1,\alpha}(\overline{\Omega})}\left( |\hat{x}-\hat{y}|^{\alpha}
+\mathrm{Lip}(\nu)\mathrm{Lip}(\Pi )|x-y|\right) .
\end{split}
\end{equation*}
On the other hand, 
\begin{equation*}
|\hat{x}-\hat{y}|\leq |\hat{x}-x|+|x-y|+|y-\hat{y}|\leq d(x)+|x-y|+d(y)\leq 3|x-y|^{\frac{1}{\alpha+1}}
\end{equation*}
as $|x-y|<\frac{\delta }{2}<1$. Therefore, 
\begin{equation}\label{caseone}
\Big\vert\frac{u(x)}{d(x)}-\frac{u(y)}{d(y)}\Big\vert\leq C_{3}\Vert u\Vert_{C^{1,\alpha}(\overline{\Omega })}|x-y|^{\beta}.  
\end{equation}
2) $d(x)>|x-y|^{\frac{1}{\alpha+1}}$. Inequality \eqref{Lipu} gives 
\begin{equation}\label{casetwo}
\begin{split}
\Big\vert\frac{u(x)}{d(x)}-\frac{u(y)}{d(y)}\Big\vert& \leq \Big\vert\frac{u(x)-u(y)}{d(x)}\Big\vert+|u(y)|\Big\vert\frac{d(x)-d(y)}{d(x)d(y)}\Big\vert\\
& \leq\mathrm{Lip}(u)\frac{|x-y|}{d(x)}+\mathrm{Lip}(u)\,d(y)\frac{|x-y|}{d(x)d(y)} \\
& \leq 2\mathrm{Lip}(u)|x-y|^{\beta}\leq 2\Vert u\Vert_{C^{1,\alpha}(\overline{\Omega })}
|x-y|^{\beta}.
\end{split}
\end{equation}
Gathering together \eqref{casezero}--\eqref{casetwo} yields \eqref{found} and completes the proof of $({\rm a}_1)$.

By Weierstrass' theorem, assertion $({\rm a}_2)$ is verified once we see that (the extension by continuity of) $d^{-1}u>0$ on $\partial\Omega$. With this aim, fix any $x_0\in\partial\Omega$. If $t>0$ is small enough then $\Pi(x_0-t\nu(x_0))=x_0$, whence $d(x_0-t\nu(x_0))=t$ thanks to \eqref{propPi}. Now, recall that $u(x_0)=0$. Using the boundary point lemma \cite[Theorem 5.5.1]{PS} one has
\begin{equation*}
\frac{u(x_0)}{d(x_0)}=\lim_{t\to 0^+}\frac{u(x_0-t\nu(x_0))}{d(x_0-t\nu(x_0))}=
\lim_{t\to 0^+}\frac{u(x_0-t\nu(x_0))-u(x_0)}{t}=-\frac{\partial u}{\partial\nu}(x_0)>0.
\end{equation*}
Since $x_0$ was arbitrary, the conclusion follows.
\end{proof}
Let $A_p:W^{1,p}_0(\Omega)\to W^{-1,p'}(\Omega)$ be the nonlinear operator stemming from the negative $p$-Laplacian, i.e.,
\begin{equation*}
\langle A_p(u),v\rangle:=\int_\Omega|\nabla u|^{p-2}\nabla u\nabla v\dx\quad\forall\, u,v\in W^{1,p}_0(\Omega).
\end{equation*}
We know \cite[Section 6.2]{GP1} that $A_p$ is bounded, continuous, and strictly monotone. 
Moreover, if $\lambda_1$ denotes the first eigenvalue of $(-\Delta_p,W^{1,p}_0(\Omega))$ then there exists a unique eigenfunction $\varphi_1$ associated with $\lambda_1$ and enjoying the properties
\begin{equation*}
\phi_1\in C^{1,\alpha}_0(\overline{\Omega})\;\;\mbox{for some}\;\; 0<\alpha<1,\quad \phi_1>0\;\;\mbox{in}\;\;\Omega,\quad
\Vert\varphi_1\Vert_p=1. 
\end{equation*}
Finally, just like in \ref{CZ} of Example \ref{slex}, a suitable differentiability result yields the following
\begin{prop}\label{locality}
The $p$-Laplacian is a strongly local operator of $W^{1,p}(\Omega)$ into $L^2_\loc(\Omega)$. More generally, given any $D\subseteq\R$ such that $|D|=0$, $v\in L^2_\loc(\Omega)$, and $u\in W^{1,p}(\Omega)$ distributional solution of $\Delta_p u = v$, one has
$$\Delta_p u(x)=0\quad\mbox{for almost every}\quad x\in u^{-1}(D). $$
\end{prop}
\begin{proof}
For $D$, $v$, $u$ as above, set $C:=u^{-1}(D)$ and $E(u):=|\nabla u|^{p-2}\nabla u$. By \cite[Lemma 1]{DGBDL} one has $\nabla u = 0$, whence $E(u)=0$, a.e. in $C$. Thanks to \cite{CM}, from $\Delta_p u \in L^2_\loc(\Omega)$ it follows $E(u)\in W^{1,2}_\loc(\Omega,\R^N)$. Using \cite[Lemma 1]{DGBDL} again we see that the Jacobian of the vector field $E(u)$ vanishes on $C$. So, in particular, $\Delta_p u(x)= \Div(E(u))(x)=0$ for almost every $x\in C$.
\end{proof}
\section{Proof of the main result}
As usual \cite{CLM}, a function $\underline{u}\in W^{1,p}(\Omega)$ is called a sub-solution to \eqref{problem} when $\underline{u}\leq 0$ on $\partial\Omega$ and
\begin{equation}\label{defsubsol}
\int_\Omega|\nabla\underline{u}|^{p-2}\nabla\underline{u}\nabla\phi\dx \leq
\int_\Omega f(\underline{u})\phi\dx\quad\forall\, \phi\in W^{1,p}_0(\Omega)_+.
\end{equation}
%
%
\begin{lemma}\label{subsollemma}
Let \ref{subsol} be satisfied. Then \eqref{problem} admits a sub-solution $\underline{u}\in C^{1,\alpha}_0(\overline{\Omega})$.
\end{lemma}
\begin{proof}
Assumption \ref{subsol} provides $\delta>0$ such that
\begin{equation}\label{estbelow}
f(s)\geq\lambda_1 s^{p-1} \quad \forall\, s\in(0,\delta).
\end{equation}
Set $\underline{u}:=k\phi_1$, where $k>0$ is so small that $\|\underline{u}\|_\infty <\frac{\delta}{2}$. Thus, $\underline{u}=0$ on $\partial\Omega$ and
$$-\Delta_p\underline{u}=\lambda_1\underline{u}^{p-1}\leq f(\underline{u})$$
through \eqref{estbelow}. This clearly implies \eqref{defsubsol}.
\end{proof}
\begin{rmk}
The condition
\begin{equation}\label{condoveru}
\|\underline{u}\|_\infty<\frac{\delta}{2}    
\end{equation}
is stronger than what we need for using \eqref{estbelow}, namely $\|\underline{u}\|_\infty<\delta$. However, it will play a crucial role in the proof of Theorem \ref{solauxprob} below; cf. \eqref{inclusionsets}.
\end{rmk}
Our next step is both truncating and regularizing the right-hand side. Let $\underline{u}$ be given by Lemma \ref{subsollemma}. Define, for every $(x,s)\in\Omega\times\R$,
\begin{equation}\label{truncation}
g(x,s):=f(\max\{\underline{u}(x),s\})
\end{equation}
as well as
\begin{equation}\label{regularization}
g_\eps(x,s):=\frac{1}{\eps}\int_{-\infty}^{+\infty} g(x,s-\xi)
\rho\left(\frac{\xi}{\eps}\right) \dxi,
\quad G_\eps(x,s):=\int_0^s g_\eps(x,t)\dt,
\end{equation}
where $\eps>0$ and
\begin{equation}\label{proprho}
\rho\in C^\infty(\R)_+,\quad {\rm supp}\,\rho\subseteq [-1,1],\quad
\int_{-\infty}^{+\infty}\rho(s)\,{\rm d}s=1.    
\end{equation} 
One evidently has
$$\essinf_{|\xi|<\eps} g(x,s-\xi)\leq g_\eps(x,s)\leq\esssup_{|\xi|<\eps} g(x,s-\xi) \quad\forall\, (x,s)\in\Omega\times\R. $$
\begin{lemma}\label{gestlemma}
Under \ref{localbound}--\ref{sublinear}, there exist $c_1>0$, $c_2\in (0,\lambda_1)$, $c_3>0$ such that
\begin{equation*}
0\leq g_\eps(x,s)\leq \esssup_{|t-s|<\eps} g(x,t)\leq c_1 \underline{u}(x)^{-\gamma}
+c_2 |s|^{p-1}+c_3 \quad\mbox{in}\quad\Omega\times\R
\end{equation*}
for every $\eps\in(0,1)$.
\end{lemma}
\begin{proof}
Hypothesis \ref{singular} yields $c_1,r>0$ fulfilling
\begin{equation}\label{fest1}
f(s)\leq c_1 s^{-\gamma} \quad \forall\, s\in(0,r).
\end{equation}
Via \ref{sublinear} we get $\hat{c}\in(0,\lambda_1)$ and $R>r$ such that
\begin{equation}\label{fest2}
f(s)\leq \hat{c} s^{p-1},\quad s\in(R,+\infty).
\end{equation}
Since, thanks to \ref{localbound},
\begin{equation}\label{fest3}
f(s)\leq M\quad\mbox{in}\quad [r,R]
\end{equation}
with appropriate $M>0$, from \eqref{fest1}--\eqref{fest3} it follows
\begin{equation}\label{fest}
f(s)\leq c_1 s^{-\gamma} + \hat{c} s^{p-1} + M \quad \forall\, s\in(0,+\infty).
\end{equation}
Pick any $\eps\in (0,1)$. Properties \eqref{proprho}, \eqref{truncation}, \eqref{fest}, besides the elementary inequality
\begin{equation*}
\hat{c}(|s|+1)^{p-1}\leq (\hat{c}+\sigma)|s|^{p-1} + C(\sigma)\, , \quad\sigma>0,
\end{equation*}
produce
\begin{equation}\label{gest}
\begin{aligned}
g_\eps(x,s) & \leq \left(\esssup_{|\xi|<\eps} g(x,s-\xi)\right)
\frac{1}{\eps}\int_{-\eps}^\eps\rho\left(\frac{\xi}{\eps}\right) \dxi \\
&= \esssup_{|t-s|<\eps} g(x,t) = \esssup_{|t-s|<\eps} f(\max\{\underline{u}(x),t\}) \\
&\leq c_1 \underline{u}(x)^{-\gamma} + \hat{c} \|\underline{u}\|_\infty^{p-1} + \hat{c} (|s|+1)^{p-1} + M \\
&\leq c_1 \underline{u}(x)^{-\gamma} + c_2 |s|^{p-1} + c_3,
\end{aligned}
\end{equation}
where $c_2\in(\hat{c},\lambda_1)$, $c_3>0$ are constants independent of $\eps$.
\end{proof}
Now, associate with \eqref{problem} the auxiliary problem
\begin{equation}\label{auxprob}
\hypertarget{auxprob}{}
\tag{${\rm P}_\eps$}
\left\{
\begin{alignedat}{2}
-\Delta_p u &= g_\eps(x,u) \quad &&\mbox{in}\;\;\Omega, \\
u&=0 \quad &&\mbox{on}\;\;\partial\Omega,
\end{alignedat}
\right.
\end{equation}
whose energy functional $J_\eps:W^{1,p}_0(\Omega)\to\R$ is
\begin{equation*}
J_\eps(u):=\frac{1}{p}\| u\|_{1,p}^p-\int_\Omega G_\eps(\cdot,u)\dx,\quad u\in W^{1,p}_0(\Omega).
\end{equation*}
\begin{thm}\label{solauxprob}
Let \ref{localbound}--\ref{sublinear} be satisfied. Then:
\begin{itemize}
\item[$({\rm a}_3)$]  For every $\eps\in(0,1)$ problem \eqref{auxprob} admits a solution $u_\eps\in W^{1,p}_0(\Omega)$.
\item[$({\rm a}_4)$] There exists $L>0$ such that $\| u_\eps\|_{1,p}\leq L$ whatever $\eps\in (0,1)$.
\end{itemize}
\end{thm}
\begin{proof}
Fix $\eps\in (0,1)$. Observe at first that $J_\eps\in C^1(X)$ because the function $g_\eps$ defined in \eqref{regularization} satisfies Carathéodory's conditions, by standard results on convolutions (see, e.g., \cite[Section 4.4]{Br}), and fulfills the growth condition \eqref{gest}. Next, conclusion $({\rm a}_2)$ of Proposition \ref{distcomparison} and Proposition \ref{hardysobolev} entail
\begin{equation}\label{hardyest}
\int_\Omega\underline{u}^{-\gamma}|u|\dx\leq l\int_\Omega d^{-\gamma}|u|\dx
\leq lK \| u\|_{1,p}\quad\forall\, u\in W^{1,p}_0(\Omega).
\end{equation}
Exploiting Lemma \ref{gestlemma}, \eqref{hardyest}, besides the Poincaré inequality, we thus arrive at
\begin{equation}\label{Jest}
\begin{aligned}
J_\eps(u) & \geq\frac{1}{p}\| u\|_{1,p}^p
-\int_\Omega\left(\int_0^{|u(x)|} g_\eps(x,t) \dt\right)\dx \\
&\geq \frac{1}{p}\| u\|_{1,p}^p
-\int_\Omega\left(c_1 \underline{u}^{-\gamma}|u|+\frac{c_2}{p} |u|^p + c_3 |u|\right)\dx \\
&\geq \frac{1}{p}\left(1-\frac{c_2}{\lambda_1}\right)\| u\|_{1,p}^p - C\| u\|_{1,p},\quad u\in W^{1,p}_0(\Omega),
\end{aligned}
\end{equation}
which yields the coercivity of $J_\eps$ (recall that $c_2<\lambda_1$). Moreover, $J_\eps$ turns out weakly sequentially lower semi-continuous. This evidently holds once
\begin{equation}\label{uscG}
u_n \rightharpoonup u\;\;\mbox{in}\;\; W^{1,p}_0(\Omega)\implies
\limsup_{n\to\infty}\int_\Omega G_\eps(\cdot,u_n)\dx\leq\int_\Omega G_\eps(\cdot,u)\dx.    
\end{equation}
To verify \eqref{uscG}, note that $(u_n)_\pm \rightharpoonup u_\pm$ in $W^{1,p}_0(\Omega)$;
cf. \cite[Lemma 2.1]{BG}. By \eqref{hardyest}, the linear functional $u\mapsto \int_\Omega \underline{u}^{-\gamma}u \dx$ is continuous on $W^{1,p}_0(\Omega)$. Therefore,
\begin{equation}\label{Gest1}
\lim_{n\to\infty}\int_\Omega\underline{u}^{-\gamma}|u_n|\dx
=\int_\Omega \underline{u}^{-\gamma}|u|\dx.
\end{equation}
Now, up to sub-sequences, one has $u_n\to u$ in $L^p(\Omega)$, $u_n(x)\to u(x)$ for almost every $x\in\Omega$, as well as 
\begin{equation}\label{inverselebesgue}
|u_n(x)|\leq w(x)\quad\mbox{a.e. in}\;\;\Omega,
\end{equation}
where $w\in L^p(\Omega)$, whence
\begin{equation}\label{Gest2}
\lim_{n\to\infty}\int_\Omega\left(\frac{c_2}{p}|u_n|^p+c_3|u_n|\right)\dx
=\int_\Omega \left(\frac{c_2}{p}|u|^p+c_3|u|\right) \dx.
\end{equation}
From \eqref{Gest1}--\eqref{Gest2}, \eqref{gest} (see also \eqref{Jest}), besides Fatou's lemma, we deduce
\begin{equation}\label{uppersemicont}
\begin{aligned}
&\int_\Omega\left(c_1\underline{u}^{-\gamma}|u|+\frac{c_2}{p}|u|^p+c_3|u|\right)\dx
-\limsup_{n\to\infty}\int_\Omega G_\eps(\cdot,u_n) \dx \\
&= \liminf_{n\to\infty}\int_\Omega\left(c_1\underline{u}^{-\gamma}|u_n|
+\frac{c_2}{p}|u_n|^p+c_3|u_n|-G_\eps(\cdot,u_n)\right) \dx \\
&\geq\int_\Omega\left(c_1\underline{u}^{-\gamma}|u|+\frac{c_2}{p}|u|^p+c_3|u|\right)\dx
-\int_\Omega G_\eps(\cdot,u) \dx,
\end{aligned}
\end{equation}
thus completing the proof of \eqref{uscG}. At this point, Weierstrass-Tonelli's theorem can be applied. So, $J_\eps$ possesses a global minimizer $u_\eps\in W^{1,p}_0(\Omega)$, and $({\rm a}_3)$ easily follows.

To show $({\rm a}_4)$, let us test \eqref{auxprob} with $u_\eps$. Lemma \ref{gestlemma}, \eqref{hardyest}, and Poincaré's inequality yield
\begin{equation*}
\begin{aligned}
\| u_\eps\|_{1,p}^p=\int_\Omega g_\eps(x,u_\eps)u_\eps\dx
\leq\int_\Omega\left(c_1 \underline{u}^{-\gamma} + c_2 |u_\eps|^{p-1} + c_3\right)|u_\eps|\dx \leq \frac{c_2}{\lambda_1}\| u_\eps\|_{1,p}^p+C\| u_\eps\|_{1,p},
\end{aligned}
\end{equation*}
where $C>0$ does not depend on $\eps$. Recalling that $c_2<\lambda_1$, a standard argument leads to the conclusion.
\end{proof}
\begin{lemma}\label{solincl}
Suppose \ref{localbound}--\ref{sublinear}. If $u_n$, $n\in\N$, denotes the solution of \hyperlink{auxprob}{$({\rm P}_{1/n})$} given by $({\rm a}_3)$ in Theorem \ref{solauxprob} for $\eps:=1/n$ then, taking a sub-sequence when necessary, $u_n\rightharpoonup u$ in $W^{1,p}_0(\Omega)$, $u\geq \underline{u}$, and
\begin{equation}\label{diffinclusion}
\hypertarget{diffinclusion}{}
\left\{
\begin{alignedat}{2}
-\Delta_p u & =v(x) \quad &&\mbox{in}\;\;\Omega, \\
u & =0\quad &&\mbox{on}\;\;\partial\Omega,
\end{alignedat}
\right.
\end{equation}
where $v$ enjoys the properties:
\begin{itemize}
\item $v\in L^q(\Omega)_+\cap W^{-1,p'}(\Omega)$, with $1<q<\min\{\frac{1}{\gamma}, \frac{p^*}{p-1}\}$;
\item $\underline{f}(u)\leq v\leq \overline{f}(u)$.
\end{itemize}
\end{lemma}
\begin{proof}
Conclusion $({\rm a}_4)$ of Theorem \ref{solauxprob} ensures that $\{u_n\} \subseteq W^{1,p}_0(\Omega)$ is bounded. Therefore, up to sub-sequences,  
$$u_n\rightharpoonup u\;\;\mbox{in}\;\; W^{1,p}_0(\Omega),\quad u_n\to u\;\;\mbox{in}\;\; L^p(\Omega),\quad u_n(x)\to u(x)\;\;\mbox{for almost every}\;\; x\in\Omega,$$
and \eqref{inverselebesgue} holds. Now, fix any $r\in\left(1,\frac{1}{\gamma}\right)$. Due to Proposition \ref{distsumm} one has $d^{-\gamma}\in L^r(\Omega)$. Proposition \ref{distcomparison} thus entails $\underline{u}^{-\gamma}\in L^r(\Omega)$. By the Sobolev embedding theorem $\{|u_n|^{p-1}\}$ turns out bounded in $L^{\frac{p^*}{p-1}}(\Omega)$. Hence, from Lemma \ref{gestlemma} it follows
\begin{equation} \label{weakconv}
g_{\eps_n}(\cdot,u_n)\rightharpoonup v \quad \mbox{in} \quad L^q(\Omega),
\end{equation}
where $q\in\left(1,\min\{\frac{1}{\gamma},\frac{p^*}{p-1}\}\right)$. A standard argument, chiefly based on Mazur's theorem \cite[Corollary 3.8]{Br} and the information $g_{\eps_n}(\cdot,u_n)\geq 0$, $n\in\N$, yields $v\geq 0$. Moreover,
\begin{equation}\label{weaklimit1}
\langle A_p(u),\phi\rangle=\lim_{n\to\infty}\langle A_p(u_n),\phi\rangle= \lim_{n\to\infty}\int_\Omega g_{\eps_n}(\cdot,u_n)\phi\dx = \int_\Omega v\phi\dx\quad \forall\,\phi\in C^\infty_c(\Omega).
\end{equation}
We claim that $v\in W^{-1,p'}(\Omega)$. Indeed, using Lemma \ref{gestlemma} and \eqref{inverselebesgue} provides
\begin{equation}\label{gdualest}
0\leq g_{\eps_n}(\cdot,u_n)\leq c_1 \underline{u}^{-\gamma} +c_2 w^{p-1} + c_3 =:\eta, \quad n\in\N,
\end{equation}
with $\eta\in W^{-1,p'}(\Omega)$ by \eqref{hardyest}. If $\phi\in W^{1,p}_0(\Omega)_+$, $\{\phi_m\}\subseteq C^\infty_c(\Omega)_+$, and $\phi_m \to \phi$ in $W^{1,p}_0(\Omega)$ then 
\begin{equation*}\begin{aligned}
\int_\Omega v\phi\dx & \leq\liminf_{m\to\infty}\int_\Omega v\phi_m\dx
=\liminf_{m\to\infty} \lim_{n\to\infty}\int_\Omega g_{\eps_n}(x,u_n)\phi_m\dx \\
&\leq \lim_{m\to\infty}\int_\Omega\eta\phi_m\dx =\int_\Omega\eta\phi\dx
\leq C\|\phi\|_{1,p}
\end{aligned}
\end{equation*}
thanks to \eqref{weaklimit1}--\eqref{gdualest} besides Fatou's lemma. For a generic $\phi\in W^{1,p}_0(\Omega)$ we still have 
\begin{equation*}
\left|\int_\Omega v\phi\dx\right|\leq\int_\Omega v|\phi| \dx\leq C\|\phi\|_{1,p}
\end{equation*}
because $|\phi|\in W^{1,p}_0(\Omega)_+$. Consequently, $v\in W^{-1,p'}(\Omega)$, as desired.

Since $C^\infty_c(\Omega)$ is dense in $W^{1,p}_0(\Omega)$, from \eqref{weaklimit1} it follows
\begin{equation}\label{weaklimit2}
\langle A_p(u),\phi\rangle=\int_\Omega v\phi\dx\quad\forall\,\phi\in W^{1,p}_0(\Omega),
\end{equation}
namely \eqref{diffinclusion} holds true. Further, the Maximum Principle \cite[Theorem 3.2.2]{PS} forces $u\geq 0$.

Next, on account of Egorov's theorem, to every $\tau>0$ there corresponds a measurable set $E_\tau\subseteq \Omega$ such that $|E_\tau|<\tau$ and $u_n\to u$ uniformly on $\Omega\setminus E_\tau$, whence for any $\sigma>0$ we can find $\nu\in\N$ fulfilling
\begin{equation*}
n>\nu\implies\eps_n<\frac{\sigma}{2},\quad\sup_{x\in\Omega\setminus E_\tau} |u_n(x)-u(x)|<\frac{\sigma}{2}.
\end{equation*}
Thus, a fortiori,
\begin{equation}\label{bounds1}
\begin{aligned}
\essinf_{|s-u(x)|<\sigma} g(x,s) &\leq  \essinf_{|s-u_n(x)|<\frac{\sigma}{2}} g(x,s) \leq g_{\eps_n}(x,u_n(x)) \\
&\leq \esssup_{|s-u_n(x)|<\frac{\sigma}{2}} g(x,s) \leq \esssup_{|s-u(x)|<\sigma} g(x,s)
\end{aligned}
\end{equation}
whenever $n>\nu$ and $x\in\Omega\setminus E_\tau$. Fix $\phi\in C^\infty_c(\Omega)_+$. 
Multiplying \eqref{bounds1} by $\phi\chi_{\Omega\setminus E_\tau}\in L^\infty(\Omega)_+$, integrating over $\Omega$, letting $n\to\infty$, and using \eqref{weakconv} we arrive at
%
\begin{equation*}
\int_{\Omega\setminus E_\tau}\essinf_{|s-u|<\sigma} g(\cdot,s)\,\phi\dx
\leq\int_{\Omega\setminus E_\tau} v\phi\dx\leq
\int_{\Omega\setminus E_\tau}\esssup_{|s-u|<\sigma} g(\cdot,s)\,\phi\dx.
\end{equation*}
Since $\phi\geq 0$ was arbitrary, this entails
\begin{equation*}
\essinf_{|s-u(x)|<\sigma} g(x,s)\leq v(x)\leq\esssup_{|s-u(x)|<\sigma} g(x,s)\quad \mbox{for almost all} \;\; x\in\Omega\setminus E_\tau.
\end{equation*}
Now, from $\lim_{\tau\to 0^+}|E_\tau|=0$ it follows 
\begin{equation}\label{bounds4}
\essinf_{|s-u(x)|<\sigma} g(x,s) \leq v(x) \leq \esssup_{|s-u(x)|<\sigma} g(x,s) \quad \mbox{a.e. in}\;\;\Omega.
\end{equation}
Suppose $\sigma\in\left(0,\frac{\delta}{2}\right)$, where $\delta>0$ satisfies \eqref{condoveru}. Then
\begin{equation}\label{inclusionsets}
]{u}(x)-\sigma,{u}(x)+\sigma[\, \subseteq\, ]-\delta,\delta[\quad \forall\, x\in\Omega(u\leq\underline{u}).
\end{equation}
Moreover, due to \eqref{truncation}, 
\begin{equation}\label{imagesets}
g(x,]-s,s[)\subseteq f([\underline{u}(x),s[)\quad\mbox{whenever}\;\; s>\underline{u}(x).
\end{equation}
Via \eqref{weaklimit2}, \eqref{bounds4}, \eqref{inclusionsets}, \eqref{imagesets} with $s=\delta$, and \eqref{estbelow} one has, after recalling that $\underline{u}=k\phi_1$ with $k>0$ (cf. Lemma \ref{subsollemma}),
\begin{equation*}
\begin{aligned}
\langle A_p & (u),(\underline{u}-u)_+\rangle=
\int_{\Omega(u\leq\underline{u})} v(\underline{u}-u)\dx\\
&\geq\int_{\Omega(u\leq\underline{u})}\left[\essinf_{|s-u(x)|<\sigma} g(x,s)\right] (\underline{u}-u)\dx \\
&\geq\int_{\Omega(u\leq\underline{u})}\left[\essinf_{|s|<\delta} g(x,s)\right] (\underline{u}-u)\dx =
\int_{\Omega(u\leq\underline{u})}\left[\essinf_{s\in[\underline{u}(x),\delta[} f(s)\right]
(\underline{u}-u)\dx \\
&\geq\lambda_1\int_{\Omega(u\leq\underline{u})}\underline{u}^{p-1} (\underline{u}-u)\dx
=\langle A_p(\underline{u}), (\underline{u}-u)_+ \rangle,
\end{aligned}
\end{equation*}
whence $\underline{u}\leq u$ by the strict monotonicity of the operator $A_p$. Finally, let $\sigma\to 0^+$ in \eqref{bounds4} and exploit \eqref{imagesets} to achieve
\begin{equation*}
\underline{f}(u(x))\leq v(x)\leq\overline{f}(u(x))\quad \mbox{for almost every} \;\; x\in\Omega.
\end{equation*}
This completes the proof.
\end{proof}
\begin{rmk}
In the context of non-smooth analysis, Lemma \ref{solincl} can be seen as an existence result for the differential inclusion
\begin{equation*}
\left\{
\begin{alignedat}{2}
-\Delta_p u & \in\partial F(u)\quad &&\mbox{in}\;\;\Omega, \\
u & =0 \quad &&\mbox{on}\;\;\partial\Omega,
\end{alignedat}
\right.
\end{equation*}
where $\partial F$ denotes the Clarke sub-differential of $F(s):=\int_0^s f(t)\dt$; see, e.g., \cite[Example 1]{LM}.
\end{rmk}
\begin{lemma}\label{regularity}
Under \ref{localbound}--\ref{sublinear}, any $u\in W^{1,p}_0(\Omega)$ such that
\begin{equation}\label{subprob}
\left\{
\begin{alignedat}{2}
-\Delta_p u & \leq \overline{f}(u) \quad && \mbox{in}\;\;\Omega, \\
u & >0 \quad && \mbox{in}\;\;\Omega, \\
u & =0 \quad &&\mbox{on}\;\;\partial\Omega,
\end{alignedat}
\right.
\end{equation}
belongs to $L^\infty(\Omega)$. If, moreover, $u$ satisfies \eqref{diffinclusion} for some $v \in L^1_\loc(\Omega)$ and $u\geq cd$ for a suitable $c>0$, then $u\in C^{1,\alpha}(\overline{\Omega})$. The corresponding $L^\infty$ and $C^{1,\alpha}$ estimates do not depend on $u$.
\end{lemma}
\begin{proof}
Fix $x\in\Omega$ fulfilling $u(x)>0$. Since $u(x)-\sigma>0$ whatever $\sigma>0$ sufficiently small, \eqref{fest} entails
\begin{equation*}\begin{aligned}
\overline{f}(u(x)) & \leq\lim_{\sigma\to 0^+}\sup_{|s-u(x)|<\sigma}\left(c_1 s^{-\gamma}
+\hat{c} s^{p-1}+ M\right) \\
& \leq\lim_{\sigma\to 0^+}\left(c_1(u(x)-\sigma)^{-\gamma}+\hat{c}(u(x)+\sigma)^{p-1}
+ M\right)=c_1 u(x)^{-\gamma}+\hat{c} u(x)^{p-1}+M.
\end{aligned}
\end{equation*}
Accordingly,
\begin{equation}
\label{festsharp}
\overline{f}(u) \leq c_1 u^{-\gamma} + \hat{c}u^{p-1} + M.
\end{equation}
Now, pick $k>1$. Choosing $(u-k)_+$ as a test function in \eqref{subprob} we obtain
\begin{equation}\label{testDG}
\int_\Omega|\nabla u|^{p-2}\nabla u\nabla (u-k)_+\dx
\leq\int_\Omega \overline{f}(u)(u-k)_+\dx,
\end{equation}
while \eqref{festsharp} leads to
\begin{equation}\label{forDG}
\begin{aligned}
\int_\Omega \overline{f}(u)(u-k)_+ \dx &\leq \int_{\Omega(u\geq k)} \left(c_1 u^{1-\gamma} + \hat{c}u^p + Mu \right) \dx \\
&\leq C \int_{\Omega(u\geq k)} u^p \dx 
\leq C \left( \|(u-k)_+\|_p^p + k^p|\Omega(u\geq k)|\right)
\end{aligned}
\end{equation}
by enlarging $C>0$ if necessary. On the other hand, from Sobolev's inequality it follows
\begin{equation}\label{Sobolev}
\int_\Omega |\nabla u|^{p-2}\nabla u\nabla (u-k)_+\dx =\|\nabla (u-k)_+\|_p^p 
\geq c\|(u-k)_+\|_{p^*}^p.
\end{equation}
Putting \eqref{forDG}--\eqref{Sobolev} into \eqref{testDG} and using \cite[Lemma 3.3]{CGL} yield $u\in L^\infty(\Omega)$.
\vskip0pt
Next, assume $-\Delta_p u=v$ for some $v\in L^1_\loc(\Omega)$ and $u\geq cd$, with suitable $c>0$. Inequality \eqref{festsharp} again produces
$$\overline{f}(u)\leq c_1 u^{-\gamma}+\hat{c} u^{p-1}+M
\leq c_1 c^{-\gamma} d^{-\gamma} + \hat{c} u^{p-1}+M,$$
whence
\begin{equation}
\label{vbound}
v=-\Delta_ p u\leq \overline{f}(u)\leq Cd^{-\gamma},
\end{equation}
because $u\in L^\infty(\Omega)$.  Through \cite[Lemma 3.1]{H} one achieves $u\in C^{1,\alpha}(\overline{\Omega})$.
\vskip0pt
To conclude, notice that both $L^\infty$ and $C^{1,\alpha}$ estimates depend only on $\Vert u\Vert_{1,p}$, besides the data of the problem. So, they are uniform in $u$ once $\|u\|_{1,p}$ depends only on such data. To check this, let us test \eqref{subprob} with $u$ and observe that, by \eqref{festsharp}, Young's inequality, and Poincaré's inequality, for every $\eps\in (0,\lambda_1-\hat{c})$ we have
$$\| u\|_{1,p}^p\leq\int_\Omega\overline{f}(u)u\dx
\leq\int_\Omega\left(c_1 u^{1-\gamma}+\hat{c}u^p+ Mu\right)\dx
\leq (\hat{c}+\eps)\|u\|_p^p+C_\eps
\leq\frac{\hat{c}+\eps}{\lambda_1}\| u\|_{1,p}^p + C_\eps.$$
Since $\frac{\hat{c}+\eps}{\lambda_1}<1$, the assertion holds.
\end{proof}
\begin{proof}[Proof of Theorem \ref{mainthm}]
The argument is patterned after \cite[Theorem 60]{LM}. We want to show that the function $u\in C^{1,\alpha}(\overline{\Omega})$ produced in Lemma \ref{solincl} (see also Lemma \ref{regularity}) turns out a strong solution of \eqref{problem}. Lemma \ref{solincl} ensures that
\begin{equation}\label{useful}
-\Delta_p u=v\in [\underline{f}(u),\overline{f}(u)],    
\end{equation}
so $v\in L^2_\loc(\Omega)$ by \eqref{vbound}. From \cite[Theorem 2.1]{CM} it thus follows $|\nabla u|^{p-2}\nabla u\in W^{1,2}_\loc(\Omega)$. Now, Proposition \ref{locality} and \ref{zeromeasure} yield $\Delta_p u = 0$ a.e. in $u^{-1}(\D_f)$. Hence, 
\begin{equation}\label{locality1}
-\Delta_p u(x)=0=f(u(x))\quad\mbox{for almost every}\;\; x\in u^{-1}(\D_f)
\end{equation}
thanks to \eqref{useful} and \ref{fzero}. On the other hand, one evidently has
\begin{equation}\label{locality2}
-\Delta_p u=f(u)\quad\mbox{a.e. in} \;\;\Omega\setminus u^{-1}(\D_f).
\end{equation}
Gathering \eqref{locality1}--\eqref{locality2} together concludes the proof.
\end{proof}
Let us finally make two examples of nonlinearities that comply with \hyperlink{Hf}{$({\rm H}_f)$}.
\begin{ex}[Non-singular case]
If $\sigma>0$ and $g:[0,1]\to\R^+_0$ is of bounded variation on $[0,1]$ then the function 
\begin{equation*}
f(s):=(g(s)+\sigma)\chi_{(0,1)}(s),\quad s\in\R^+,
\end{equation*}
satisfies \ref{localbound}--\ref{fzero}. In particular, Jordan's decomposition theorem \cite[Corollary 2.7]{WZ} ensures that $\mathcal{D}_f$ turns out at most countable, whence $|\mathcal{D}_f|=0$.
\end{ex}
\begin{ex}[Singular case]
Suppose $\lambda\in[0,\lambda_1)$ and denote by $[s]$ the integer part of $s\in\R$. Then the function 
%
\begin{equation*}
f(s):=\left[\frac{1}{s}\right]^{\gamma}\chi_{(0,1)}(s)+\lambda[s]^{p-1}\chi_{[1,+\infty)}(s) \quad\forall\, s\in\R^+
%
\end{equation*}
(where, as above, $0<\gamma<1<p<N$) fulfills \hyperlink{Hf}{$({\rm H}_f)$}. Notice that $f$ has countably many discontinuity points accumulating at $s=0$. So, any solution to \eqref{problem} must cross $\D_f$.
\end{ex}
\section*{Acknowledgments}
\noindent
This study was partly funded by: Research project of MIUR (Italian Ministry of Education, University and Research) Prin 2022 {\it Nonlinear differential problems with applications to real phenomena} (Grant No. 2022ZXZTN2).

The authors are members of the {\em Gruppo Nazionale per l'Analisi Matematica, la Probabilit\`a e le loro Applicazioni}
(GNAMPA) of the {\em Istituto Nazionale di Alta Matematica} (INdAM).

\begin{small}

\end{small}
\end{document}